\documentclass[twocolumn, amsthm]{autart}    %

\usepackage{graphicx} %

\usepackage{amsmath} %
\usepackage{dsfont}
\usepackage{algpseudocode}
\usepackage{algorithm}
\usepackage{mathtools}
\usepackage[scr]{rsfso}
\usepackage{array, amssymb, amscd}
\usepackage{url}
\usepackage[percent]{overpic}
\usepackage{cite} %

\DeclareMathOperator*{\argmin}{arg\,min}
\DeclareMathOperator*{\minimize}{minimize}

\usepackage[acronym]{glossaries} 

\usepackage{color}

\newcommand{\sREV}{\color{black}}
\newcommand{\eREV}{\color{black}}

\newtheorem{definition}{Definition}
\newtheorem{assumption}{Assumption}
\newtheorem{proposition}{Proposition}

\newtheorem{remark}{Remark}
\newtheorem{problem}{Problem}

\newcommand{\set}[1]{\left\{ #1\right\}}

\newcommand{\X}{\mathrm{X}}

\newcommand{\eps}{\epsilon}

\newcommand{\er}[2]{\mathrm{er}(#1,#2)}
\newcommand{\erhat}[2]{\widehat{\mathrm{er}}_m(#1,#2)}

\renewcommand{\phi}{\varphi}

\newcommand{\1}{\mathds{1}}

\newcommand{\Prb}{\mathbb{P}}

\newcommand{\mass}{m}%
\newcommand{\HH}{\mathcal{H}}

\newcommand{\MP}[1]{{\color{black}#1}}

\newacronym{milp}{MILP}{Mixed Integer Linear Program}
\newacronym{aeb}{AEB}{Autonomous Emergency Braking}
\newacronym{vc}{VC}{Vapnik-Chervonenkis}
\newacronym{pac}{PAC}{probably approximately correct}

\begin{document}

\begin{frontmatter}

\title{Finite sample learning of moving targets}
\thanks[footnoteinfo]{This work was partially supported by MUR under the  PRIN 2022 project “The Scenario Approach for Control and Non-Convex Design” (project number D53D23001440006) and by FAIR (Future Artificial Intelligence Research) project, funded by the NextGenerationEU program within the PNRR-PE-AI scheme (M4C2, Investment 1.3, Line on Artificial Intelligence).
}
\author[UOXF-CS]{Nikolaus Vertovec}\ead{nikolaus.vertovec@st-hughs.ox.ac.uk},    %
\author[UOXF-ENG]{Kostas Margellos}\ead{kostas.margellos@eng.ox.ac.uk},   %
\author[POLIMI]{Maria Prandini}\ead{maria.prandini@polimi.it}  %

\address[UOXF-CS]{Department of Computer Science, University of Oxford, OX1 3PJ, UK} 
\address[UOXF-ENG]{Department of Engineering Science, University of Oxford, OX1 3PJ, UK}  %
\address[POLIMI]{Department of Electronics, Information and Bioengineering, Politecnico di Milano, Milano 20133, Italy}  %
          

\begin{abstract}                         
   We consider a moving target that we seek to learn from samples. Our results extend randomized techniques developed in control and optimization for a constant target to the case where the target is changing. We derive a novel bound on the number of samples that are required to construct a probably approximately correct (PAC) estimate of the target. Furthermore, when the moving target is a convex polytope, we provide a constructive method of generating the PAC estimate using a mixed integer linear program (MILP). The proposed method is demonstrated on an application to autonomous emergency braking. 
\end{abstract}

\end{frontmatter}

\section{Introduction}
The use of probabilistic and randomized methods to analyze and design systems affected by uncertainty has long been a key research area within the control community. Early attempts at dealing with uncertainty were focused on stochastic approaches with later research focusing on \emph{worst-case} settings. Probabilistic approaches to robustness emerged to alleviate conservatism of worst-case considerations by resorting to probabilistic information. This rapprochement between more traditional stochastic and robust paradigms facilitates uncertainty quantification based on data. To this end, we consider algorithms based on uncertainty randomization known as \emph{randomized algorithms}~\cite{Tempo2005}, which allow us to apply tools from statistical learning theory based on \gls{vc} theory to control~\cite{Vidyasagar2003, Alamo2009, Tempo2005}. 
In general, these developments can be cast as binary classification problems with the main focus being the provision of finite-sample complexity bounds. \Gls{vc} theoretic techniques require the so called \Gls{vc} dimension to be finite. The computation of the \Gls{vc} dimension is in general a difficult task for generic optimization problems. Under a convexity assumption, the so-called scenario approach has offered a theoretically sound and efficient methodology to provide \emph{a-priori} probabilistic feasibility guarantees for uncertain optimization programs, with uncertainty represented by means of scenarios and without resorting to \Gls{vc} theory \cite{Calafiore2006, Campi2008, Campi2009, Calafiore2010, Campi2018}. 
These developments have been recently extended to the non-convex case, however, they typically involve \emph{a posteriori} guarantees \cite{Campi2016, Garatti2022}.
Applications and sample complexity bounds of the aforementioned methodologies to control synthesis problems have been demonstrated in~\cite{Cloete2025, Dean2019, Campi2018a, Tempo2005}, while notable extensions involve trading feasibility to performance \cite{Campi2010, Romao2023, Romao2023a}, applications in game theory \cite{Fele2021}, and sequential methods~\cite{Tempo2005}. Connections between the scenario approach and statistical learning theory based on the notion of compression have been provided in \cite{Margellos2015, Campi2023}. 

The aforementioned approaches can be considered in the context of learning an unknown labeling mechanism, whereby we independently draw $m$ samples from a domain $\X \subseteq \mathbb{R}^n$, according to some possibly unknown probability distribution $\mathbb{P}$. Each sample is assigned a $\{0,1\}$-valued label according to an unknown \emph{target} labeling function, $f$. The learning problem involves characterizing sample complexity bounds for $m$, such that we can generate a hypothesis $h$ based on the labeled $m$-multisample that, with a prescribed confidence $1-\delta$, provides the same labeling with the target function when it comes to a new sample $x$ up to a predefined accuracy level $\epsilon$, i.e.,
\begin{align}
&\mathbb{P}^m\big\{(x_1,\ldots,x_m) \in \X^m:~ \nonumber \\
&\mathbb{P}\{x \in \X:~ h(x) \neq f(x)\} \leq \epsilon \big\} \geq 1-\delta, \label{eqn:track_og}
\end{align}
where $\mathbb{P}^m$ is the product probability measure.
An algorithm that generates a hypothesis satisfying the above statement is said to be \gls{pac} to accuracy $\eps$ if the left side of \eqref{eqn:track_og} approaches $1$ as $m\rightarrow \infty$ \cite[pg. 56]{Vidyasagar2003}. We will refer to the labeling mechanism as being \gls{pac} learnable to accuracy $\eps$ if there exists an algorithm that is \gls{pac} to accuracy $\eps$.

In this paper, we will study a similar problem of finding a hypothesis satisfying~\eqref{eqn:track_og}, however, with the notable difference that we consider a \emph{tracking problem} where the unknown labeling function is changing after each drawn sample. In light of this labeling mechanism changing in a structured manner as specified in the sequel, we will consider both the construction of the hypothesis as well as the minimum number of samples that are necessary, so as to, with a certain confidence, provide probabilistic bounds on the event of the hypothesis disagreeing with the subsequently received label. A similar tracking problem with an alternative structure of change imposed on the target is considered in \cite{Kuh1990, Barve1997, Long1999, Bartlett2000, Crammer2010}. Similar to the structure considered in this paper, \cite{Helmbold1994} considers a setting that allows for variations in the change between samples. In \cite{Long1999} the distribution according to which samples are drawn is also considered to be changing, while recent work, such as \cite{Hanneke2015}, has considered adapting to a variable rate of change of the target concept.

We first provide a formal mathematical formulation of the tracking problem considered in this paper in Section~\ref{sec:2:LearningMovingTargets}. Our main contributions can be summarized as follows:
\begin{enumerate}
    \item In Section~\ref{sec:3:FiniteSampleProbabilistic} we provide \emph{a-priori} bounds on the minimum number of samples needed to generate a \gls{pac} to accuracy $\eps$ hypothesis. This analysis capitalizes on the aforementioned references, and in particular the work of \cite{Helmbold1994}. However, we re-approach this formulation providing a \gls{pac}-type of result (that involves two layers of probability) rather than an expected value assessment. We also provide a remedy for a mathematical omission in the analysis of \cite{Helmbold1994}. 
    \item \MP{In Section~\ref{sec:4:hypothesisComputation} we provide a constructive method of generating a hypothesis from a finite set of samples using a \gls{milp} when the class of targets is that of convex polytopes. Note that the analysis in all aforementioned references is of existential nature, and this constitutes the first constructive approach for a hypothesis that enjoys such tracking properties.}
\end{enumerate}
We demonstrate numerically our theoretical results in Section~\ref{sec:4:RevisistedExample} on a case study that involves autonomous emergency braking and discuss practical improvements to excluding samples from consideration in the \gls{milp}. Finally, Section~\ref{sec:5:conclusion} provides some concluding remarks.

\section{Learning moving targets} \label{sec:2:LearningMovingTargets}
\subsection{Problem statement}
We consider the problem of learning a labeling mechanism that is changing in a structured manner (this structure will be specified in the sequel). To this end, we follow a sample-based approach, where each sample $x$ is generated independently from a domain $\X \subseteq \mathbb{R}^n$, endowed with a $\sigma$-algebra $\mathcal{X}$. Let $\mathbb{P}$ denote the fixed (potentially unknown) probability measure over $\mathcal{X}$.  
We refer to $(x_1,\ldots,x_m) \in \X^m$ as an $m$-multisample, where its elements $x_i \in \X$ are independently and identically distributed (i.i.d.) according to $\mathbb{P}$. For each $i=1,\ldots,m$, let $f_i(\cdot):~ \X \to \{0,1\}$ be a $\{0,1\}$-valued labeling function, referred to as a \emph{target function}. %

In our setting, each sample $x_i$ is labeled according to target function $f_i$, $i=1,\ldots,m$, giving rise to the \emph{labeled} $m$-multisample $\{(x_1,f_1(x_1)),\ldots, (x_m,f_m(x_m)) \}$. Notice that each sample is labeled by means of a different target function. As we consider the target functions to be unknown, we only have access to the labels of specific samples, namely $\{f_i(x_i)\}_{i=1}^m$. A natural question that we seek to answer is whether we can construct a labeling mechanism $h_m(\cdot):~ \X \to \{0,1\}$ %
that correctly (with a certain probability) predicts the label that would be assigned to the next sample $x$ by the unknown target function $f_{m+1}$. In other words, we seek to provide probabilistic guarantees that $h_m(x)=f_{m+1}(x)$, where $h_m$ is referred to as a \emph{hypothesis} and constitutes an approximation/prediction of $f_{m+1}$. Notice that we introduce the subscript $m$ to our hypothesis to highlight that this is constructed on the basis of the labeled $m$-multisample.
We refer to this problem, pictorially illustrated in Figure~\ref{fig:problem_statement}, as a tracking problem, as we seek to track a moving labeling mechanism. 

\begin{figure}[t!]
 	\centering
	\includegraphics[width=\columnwidth]{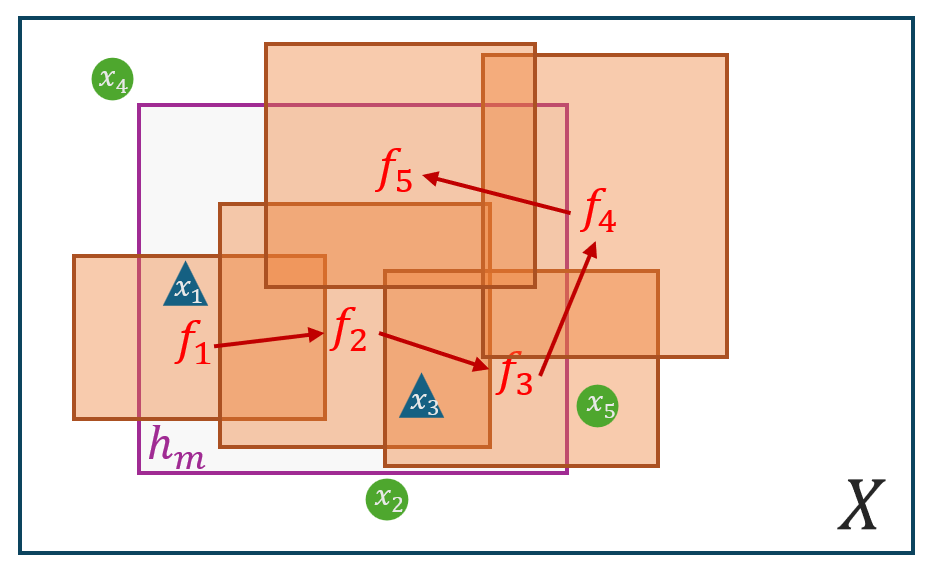}
	\caption{At each iteration, we receive a single sample along with a $\{0,1\}$-valued label. To illustrate this, consider the labeling mechanism as an indicator function over the orange set. The orange set will change between each drawn sample (we illustrate this by depicting the orange sets across multiple iterations). The green circles indicate a $0$-label, while the blue triangles represent a $1$-label. We seek to find a hypothesis on the basis of the labeling $\{(x_1,f_1(x_1)),\ldots, (x_m,f_m(x_m)) \}$ that, with certain confidence, will agree with the subsequent (unknown) target function $f_{m+1}$ on a new sample. We depict an example of such a hypothesis with the purple rectangle.}\label{fig:problem_statement}
\end{figure}

While the target functions are considered to be unknown we will make the following assumption on the target and hypotheses function class, 
\MP{whose richness as a class of $\{0,1\}$-valued labeling functions can be defined in terms of the Vapnik-Chervonenkis (VC) dimension\footnote{\MP{Given $S_m=(x_1,\dots,x_m)\in X^m$ and a class $\mathcal F$ of binary-valued labeling functions, let $F_{S_m}=\{(f(x_1),\dots, f(x_m)),\ f\in \mathcal F\}$ be the set of all possible labeling of $S_m$. Then the VC dimension of $\mathcal F$ is the smallest $m$ such that $\sup_{S_m\in X^m}|F_{S_m}|=2^m$.}} \cite{Vidyasagar2003}.}
\begin{assumption}
\label{assum:H}
    All target and hypotheses functions belong to the same class $\HH$, i.e., $f_1,\ldots,f_m,f_{m+1},h_m \in \HH$, and $\HH$ is assumed to be known. We further assume that $\HH$ has a finite $\gls{vc}$ dimension.
\end{assumption}
\begin{remark}
In Section~\ref{sec:4:hypothesisComputation} we will consider $\HH$ to be the class of non-empty convex polytopes with a certain maximum number of facets, but make no such restriction for the main results of Section~\ref{sec:3:FiniteSampleProbabilistic}.
\end{remark}

We formalize the tracking problem below.
\begin{problem}[Tracking Problem] \label{problem}
Let $\eps, \delta \in (0,1)$ be any fixed accuracy and confidence level, respectively. Determine $m_0(\epsilon,\delta)$ such that for any number of labeled samples $m \geq m_0(\epsilon,\delta)$, namely, $\big \{(x_1,f_1(x_1)),\ldots, (x_m,f_m(x_m))\big \}$, we can construct a hypothesis $h_m \in \HH$ such that 
\begin{align}
&\mathbb{P}^m\big\{(x_1,\ldots,x_m) \in \X^m:~ \nonumber \\
&\mathbb{P}\{x \in \X:~ h_m(x) \neq f_{m+1}(x)\} \leq \eps_0+\epsilon \big\} \geq 1-\delta, \label{eqn:track}
\end{align}
where $\eps_0 \in (0,1)$.
\end{problem}
In words, with confidence at least $1-\delta$, the probability that the constructed hypothesis $h_m$ produces a label for a new sample $x$ that does not agree with the target function $f_{m+1}$ is at most $\eps_0+\epsilon$. Notice that the statement we seek to provide is within the realm of \gls{pac} learning. Yet unlike more standard \gls{pac} statements, the accuracy is deteriorated by $\eps_0$; this is not user-chosen but rather depends on how the target function is moving. We specify this in the next section and show that its presence is the price to pay for providing such statements for moving targets, while $\eps_0=0$ for the specific case of a constant target.

\subsection{Mathematical preliminaries and assumptions} \label{sec2.2:Ass}
To simplify notation, for any labeling functions $f, h$ we define their probabilistic and empirical disagreement, respectively, as
\begin{align}
    \er{f}{h} & \coloneqq \Prb\{x\in \X : h(x) \neq f(x)\}, \\
    \erhat{f}{h} & \coloneqq \frac{1}{m} \sum_{i=1}^m | f(x_i) -  h(x_i)|, \label{eqn:erhat}
\end{align}
where the empirical disagreement is computed on an $m$-multisample $\{(x_1,f_1(x_1)),\ldots, (x_m,f_m(x_m)) \}$, hence we introduce the subscript $m$ in the definition of $\erhat{\cdot}{\cdot}$ to emphasize this dependence. Notice that in \eqref{eqn:erhat}, $| f(x_i) -  h(x_i)| = 1$ if $f, h$ disagree on $x_i$, and zero otherwise. Under these definitions, for $\epsilon, \delta \in (0,1)$, the statement of \eqref{eqn:track} can be equivalently written as $\mathbb{P}^m\{(x_1,\ldots,x_m) \in \X^m:~ \er{f_{m+1}}{h_m} \leq \eps_0+\epsilon \} \geq 1-\delta$.

We first provide some preliminary results that will be invoked in the subsequent developments. Proposition \ref{prop:hoeffding} below is a direct consequence of Hoeffding's inequality (see e.g., \cite{Hoeffding1963}, \cite{Tempo2005}). 
\begin{proposition}\label{prop:hoeffding}
Let $p_1,\ldots,p_m \in [0,1]$, and 
consider independent Bernoulli random variables $Y_1,\ldots,Y_m$ such that $\Prb\{Y_i = 1\} = p_i$ and $\Prb\{Y_i = 0\} = 1-p_i$, for all $i=1,\ldots,m$.
For any $\tau > 0$ we then have that
\begin{align}
    \Prb^m\big\{\sum_{i = 1}^m Y_i -\sum_{i = 1}^m p_i > \tau \big \} \leq e^{-\frac{2 \tau^2}{m}}. \label{eqn:hoeffding}
\end{align}
\end{proposition}

The following result is a PAC-type bound that holds for any target function $f \in \HH$. This is \cite[Theorem 7]{Alamo2009} adapted to our notation.
\begin{thm}\label{thm:1}
Fix $\epsilon, \delta \in (0,1)$ and $\rho \in [0,1)$. Fix any $f \in \HH$, and denote by $d$ the \gls{vc} dimension of $\HH$. For any
\begin{align}
m \geq \frac{5(\rho + \eps)}{\eps^2} \Big( \ln\frac{4}{\delta} + d \ln\frac{40 (\rho +\eps)}{\eps^2}\Big) \label{eqn:bound_m}
\end{align}
we have that
\begin{align}
    \Prb^m \{&(x_1,\ldots,x_m) \in X^m:~ \exists h \in \HH \text{ such that } \nonumber \\ &\erhat{f}{h} \leq \rho \text{ and }\er{f}{h} > \rho + \eps\} \leq \delta. \label{eqn:thm:1}
\end{align}
\end{thm}
In words, Theorem \ref{thm:1} states 
that the probability that there exists a hypothesis such that its empirical error $\erhat{f}{h}$ with the target function is at most $\rho$ but the actual error
$\er{f}{h}$ is higher than $\rho+\epsilon$, is at most equal to $\delta$ (which is typically selected to be small).
Note that unlike the tracking problems presented in \cite{Helmbold1994}, we consider two levels of probability rather than an expected value assessment.

For the subsequent developments we consider target functions that exhibit the following structure on the way the labeling is changing, i.e., the target is moving. 
\begin{assumption}
\label{assum:mu}
Let $f_1,\ldots,f_m,f_{m+1} \in \HH$, and consider $\underline{\mu}, \overline{\mu} \in (0,1)$ with $\underline{\mu} \leq \overline{\mu}$. We assume that the average probability of disagreement of the previous labels with the label $f_{m+1}$, denoted by 
\begin{equation} \label{eq:mu}
    \mu = \frac{1}{m}\sum_{i=1}^m \er{f_i}{f_{m+1}},
\end{equation}is bounded such that $\underline{\mu} \leq \mu \leq \overline{\mu}$.
\end{assumption}
Assumption \ref{assum:mu} implies that the target sets are changing but we impose a restriction (both upper and lower limits) on the probability that the labeling they produce changes. We refer to $\underline{\mu}, \overline{\mu}$ as the minimum and maximum, respectively, target change.

\section{Finite sample probabilistic certificates}\label{sec:3:FiniteSampleProbabilistic}

\subsection{Main result} \label{sec:apriori}
Problem \ref{problem} requires obtaining finite sample complexity bounds such that a hypothesis $h_m$ constructed on the basis of a labeled $m$-multisample tracks (probabilistically) the moving target function. In this section, we show that this is the case for hypotheses in minimal empirical disagreement on the $m$-multisample. We formalize the set of such hypotheses in the definition below.

\begin{definition}[minimal disagreement]\label{def:mindis}
Consider a labeled $m$-multisample $\{(x_1,f_1(x_1)),\ldots, (x_m,f_m(x_m)) \}$. We refer to the set 
\begin{equation}\label{eqn:mindis}
    M_m \coloneqq \argmin_{h \in \HH} \frac{1}{m} \sum_{i=1}^m | f_i(x_i) -  h(x_i)|
\end{equation}
as the set of hypotheses in $\HH$ that minimize the empirical error with the labeled $m$-multisample. We then say that any $h \in M_m$ is in minimal disagreement with $f_1,\ldots,f_m$.
\end{definition}

\sREV As we consider the target function and hypothesis to be a $\{0,1\}$-valued labeling functions and $\mathcal{H}$ is assumed to have finite VC dimension (Assumption \ref{assum:H}), by Sauer’s Lemma~\cite[pg.~124]{Vidyasagar2003}, the number of distinct labelings of $N$ samples is at most $\sum_{i=0}^d \binom{N}{i}$, where $d$ denotes the VC dimension. Since we are minimizing a discrete function over a finite set of labelings in $\{0,1\}^N$, a minimum must be attained, and hence $M_m$ is non-empty. \eREV

Let $h_m$ be a hypothesis in minimal disagreement with $f_1,\ldots,f_m$. We show that for this particular hypothesis choice, we can provide an answer to Problem \ref{problem}, with $\eps_0 = 4\overline{\mu}$. We formalize this in the next theorem, which is the main result of this section.

\begin{thm}
    \label{thm:main}
    Fix $\eps, \delta \in (0,1)$. 
    Denote by $d$ the \gls{vc} dimension of $\HH$, and consider Assumption~\ref{assum:mu} with $\overline{\mu} < \frac{1}{4}$.
    If we choose $m \geq m_0(\eps,\delta)$, where
    \begin{align}
            m_0(\eps,\delta) = &\max\Big \{ \frac{1}{2\underline{\mu}^2} \ln{\frac{2}{\delta}}, \nonumber \\
            & \frac{5 (4\overline{\mu}+\eps)}{\eps^2} \Big( \ln\frac{8}{\delta} + d \ln\frac{40 (4\overline{\mu}+\eps)}{\eps^2}\Big) \Big \}, \label{eqn:m0}
    \end{align}
    we then have that for any $h_m \in M_m$,
    \begin{align}
        \mathbb{P}^m\{(x_1,\ldots,x_m) &\in \X^m:~ \nonumber\\
        &\er{f_{m+1}}{h_m} \leq 4\overline{\mu}+\eps \} \geq 1-\delta. \label{eqn:thm2}
    \end{align}
\end{thm}

\begin{proof}
Fix any $\eps, \delta \in (0,1)$.
We define the following events:
\begin{align}
&E = \{(x_1,\ldots,x_m) \in \X^m :~ \er{f_{m+1}}{h_m} > 4\overline{\mu}+\eps\}, \nonumber \\
&A = \{ (x_1,\ldots,x_m) \in \X^m :~ \nonumber \\
&\hspace{2.5cm}\frac{1}{m}\sum_{i=1}^m |f_i(x_i)- f_{m+1}(x_i)| > 2\mu \}.
\end{align}
$A$ is an approximation set as it includes the $m$-multisamples for which the empirical average disagreement $\frac{1}{m}\sum_{i=1}^m |f_i(x_i)- f_{m+1}(x_i)|$ is at least twice as big as the actual average disagreement $\mu$. $E$ plays the role of the error set, as by its definition, $\mathbb{P}^m\{E\}\leq \delta$ is the complementary statement to that of \eqref{eqn:thm2}.

We can bound $\Prb^m\{E\}$ as
    \begin{align}
        \Prb^m\{E\} & = \Prb^m\{E \cap A\} + \Prb^m\{E \cap \overline{A}\} \nonumber \\
        & \leq \Prb^m\{A\} + \Prb^m\{E \cap \overline{A}\},
    \label{eqn:total}
    \end{align}
where $\overline{A}$ denotes the complement of $A$. The inequality is since $\Prb^m\{E \cap A\} \leq \Prb^m\{A\}$.
To show \eqref{eqn:thm2}, we can equivalently establish that  $\mathbb{P}^m\{E\} \leq \delta$. To achieve this, it suffices to show that $\Prb^m\{A\} \leq \delta/2$ and $\Prb^m\{E \cap \overline{A}\} \leq \delta/2$\footnote{Splitting the confidence equally between these two terms is not necessary; further optimizing the split would have minor effect on the final sample size bound as the confidence appears inside the logarithm. As such we do not pursue this here to simplify the analysis.}.

\emph{Case $\Prb^m\{A\} \leq \delta/2$:} For each $i=1,\ldots,m$, set $Y_i = |f_i(x_i)- f_{m+1}(x_i)|$ and $p_i = \er{f_i}{f_{m+1}}$ so that $\Prb\{Y_i = 1\} = p_i$ and $\Prb\{Y_i = 0\} = 1-p_i$. Notice that $Y_1,\ldots, Y_m$ are independent Bernoulli random variables, and by \eqref{eq:mu}, $\sum_{i=1}^m p_i = m\mu$. Under this variables assignment, and selecting $\tau = m\mu$, $\mathbb{P}^m\{A\}$ coincides with the left-hand side of \eqref{eqn:hoeffding}. We then have that
\begin{align}
\mathbb{P}^m\{A\} \leq e^{-2 m \mu^2} \leq e^{-2 m \underline{\mu}^2}, \label{eqn:bug}
\end{align}
where the first inequality is due to Proposition \ref{prop:hoeffding}, and the second one is since $\mu \geq \underline{\mu}$ by Assumption \ref{assum:mu}.

By inspection of \eqref{eqn:bug}, to ensure that $\Prb^m\{A\} \leq \delta/2$, it suffices to show that $e^{-2 m \underline{\mu}^2} \leq \delta/2$. By taking the logarithm making $m$ the argument, we conclude that if 
\begin{align}
m \geq \frac{1}{2\underline{\mu}^2} \ln{\frac{2}{\delta}} ~\Longrightarrow ~ \Prb^m\{A\} \leq \frac{\delta}{2}. \label{eqn:bound1}
\end{align}

\emph{Case $\Prb^m\{E \cap \overline{A}\} \leq \delta/2$:}  
We have that
    \begin{align}
        &\Prb^m\{E \cap \overline{A}\} \nonumber \\
        &\leq \Prb^m\{(x_1,\ldots,x_m)  \in \X^m:~
        \er{f_{m+1}}{h_m} > 4\overline{\mu}+\eps \nonumber \\
        & \hspace{1.5cm}\text{and } \frac{1}{m} \sum_{i=1}^m |f_i(x_i)- f_{m+1}(x_i)| \leq 2\overline{\mu}\}, \label{eqn:event}
    \end{align}
where the first statement in the right-hand side of \eqref{eqn:event} is the event $E$, and the second one encompasses $\overline{A}$. To see the latter, notice that $\overline{A}$ requires $\frac{1}{m}\sum_{i=1}^m |f_i(x_i)- f_{m+1}(x_i)| \leq 2\mu$, and $\mu \leq \overline{\mu}$ due to Assumption \ref{assum:mu}. 

We have assumed that for any $m$-multisample, $h_m$ is chosen from $M_m$. By \eqref{eqn:mindis}, since $h_m \in M_m$, $\sum_{i=1}^m |f_i(x_i)- h_m(x_i)| \leq \sum_{i=1}^m |f_i(x_i)- h(x_i)|$ for any $h\in \HH$. However, since we also have that $f_{m+1} \in \HH$, we have that for any $m$-multisample,
    \begin{equation}
    \sum_{i=1}^m |f_i(x_i)- h_m(x_i)| \leq \sum_{i=1}^m |f_i(x_i)- f_{m+1}(x_i)|.  \label{eqn:f_h}
    \end{equation}
Moreover, we have that
\begin{align}
&\erhat{f_{m+1}}{h_m} = \frac{1}{m} \sum_{i=1}^m |f_{m+1}(x_i)- h_m(x_i)| \nonumber \\
&\leq \frac{1}{m} \sum_{i=1}^m |f_i(x_i)- f_{m+1}(x_i)| +  \frac{1}{m} \sum_{i=1}^m |f_i(x_i)- h_m(x_i)|\nonumber \\
&\leq \frac{2}{m} \sum_{i=1}^m |f_i(x_i)- f_{m+1}(x_i)|,
\label{eqn:triangle}
\end{align}
where the equality is due to \eqref{eqn:erhat}, and the first inequality is by adding and subtracting $f_i(x_i)$ in each term in the summation and applying the triangle inequality. The last inequality is due to \eqref{eqn:f_h}. 

Since \eqref{eqn:triangle}
holds for any $m$-multisample, we have that
\begin{align}
\{&(x_1,\ldots,x_m) \in \X^m: \frac{1}{m} \sum_{i=1}^m |f_i(x_i)- f_{m+1}(x_i)| \leq 2\overline{\mu}\} \nonumber \\
&\subseteq \{(x_1,\ldots,x_m) \in \X^m:  \erhat{f_{m+1}}{h_m} \leq 4\overline{\mu}\}. \label{eqn:incl}
\end{align}
As a result, by \eqref{eqn:event} and \eqref{eqn:incl} we obtain
    \begin{align}
        &\Prb^m\{E \cap \overline{A}\} \nonumber \\
        &\leq \Prb^m\{(x_1,\ldots,x_m)  \in \X^m:~
        \er{f_{m+1}}{h_m} > 4\overline{\mu}+\eps \nonumber \\
        & \hspace{1.5cm}\text{and } \erhat{f_{m+1}}{h_m} \leq 4 \overline{\mu}\}.\label{eqn:event1}
    \end{align}

Notice that \eqref{eqn:event1} takes the form of \eqref{eqn:thm:1}, with $f_{m+1}$, $h_m$ and $4\overline{\mu}$ in place of $f$, $h$ and $\rho$, respectively. Theorem \ref{thm:1} with $\delta/2$ in place of $\delta$ implies that 
\begin{align}
m &\geq \frac{5 (4\overline{\mu}+\eps)}{\eps^2} \Big( \ln\frac{8}{\delta} + d \ln\frac{40 (4\overline{\mu}+\eps)}{\eps^2}\Big) \nonumber \\
&
\Longrightarrow~ \Prb^m\{E \cap \overline{A}\} \leq \frac{\delta}{2}. \label{eqn:bound2} 
\end{align}

By \eqref{eqn:bound1} and \eqref{eqn:bound2}, we obtain that if $m \geq m_0(\eps,\delta)$, where $m_0(\eps,\delta)$ is as in \eqref{eqn:m0}, we have that $\Prb^m\{E\} \leq \delta$, thus concluding the proof. 
\end{proof}
The proof of Theorem~\ref{thm:main} is inspired by \cite[Theorem 1]{Helmbold1994}. However, the result therein does not involve two layers of probability and effectively provides a bound on the expectation of the probability of incorrectly tracking the target. Moreover, only an upper bound on the target change is considered in \cite{Helmbold1994}. This is due to the fact that a term similar to $e^{-2 m\mu^2}$ was bounded by $e^{-2m \overline{\mu}^2}$, which is, however, not valid as $\mu \leq \overline{\mu}$. Here we correct this issue by introducing a lower bound on the target change, resulting in equation \eqref{eqn:bug}.

The sample size bound in \eqref{eqn:m0} depends polynomially on $1/\eps$ and logarithmically on $\delta$. This implies that we could make the confidence $1-\delta$ high without an unaffordable increase on the number of samples required. 
Figure \ref{fig:mu_bnds} illustrates the number of samples as a function of $\eps$. The color code corresponds to different values of $\underline{\mu}, \overline{\mu}$.
It should be noted that the overall accuracy level for the prediction properties of our hypothesis is $4\overline{\mu} + \eps$. Even though $\eps$ is user-chosen, $\overline{\mu}$ is a property of the target, and as such the labeling mechanism is considered to be \gls{pac} learnable to accuracy $4\overline{\mu}$. Therefore, insightful accuracy levels can be achieved if $\overline{\mu}$ is relatively low, i.e., for moderately changing target functions.

\begin{remark}[Effect of $\underline{\mu}, \overline{\mu}$]\label{rem:mu}
As evident from \eqref{eqn:m0}, the minimum number of samples that need to be generated is the maximum of two terms: the first one depends only on the minimum target change $\underline{\mu}$, while the second one depends on the maximum target change $\overline{\mu}$. These two sample size bounds that comprise $m_0(\eps,\delta)$ emanate from bounding  
\begin{enumerate}
    \item the event $A$, that the empirical average disagreement $\frac{1}{m}\sum_{i=1}^m |f_i(x_i)- f_{m+1}(x_i)|$ is at least twice as big as the actual average disagreement $\mu$. Bounding this term is responsible for the first sample size bound in \eqref{eqn:m0}.
    \item the event $E \cap \overline{A}$, that the empirical average disagreement is less than twice as the actual average disagreement $\mu$, yet that the true probability of disagreement between the hypothesis, $h_m$, and the subsequent label, $f_{m+1}$, is more than $4\overline{\mu}+\eps$. Bounding this term is responsible for the second sample size bound in \eqref{eqn:m0}.
\end{enumerate}  
With reference to Figure \ref{fig:mu_bnds}, for high values of $\epsilon$ the first sample size bound in \eqref{eqn:m0} dominates, which is independent of $\epsilon$, hence that part of each curve is constant. On the contrary, for lower values of $\epsilon$ the second sample size bound in \eqref{eqn:m0} becomes the dominant one.

If $\underline{\mu}$ is sufficiently low (the target could move slowly), then the first sample size bound in \eqref{eqn:m0} dominates. Intuitively, this implies that if the target could move slowly, then learning the actual probability of change from the empirical one (this is encoded in the definition of the event $A$) requires more samples, as with few samples we might get misleading results due to observing a faster target change than the true average change in the target. With reference to Figure \ref{fig:mu_bnds}, the minimum number of samples required increases as $\underline{\mu}$ decreases (compare the constant part of the curves). 

If we now allow for a large change of the target, encoded by a large $\overline{\mu}$, then the second sample size bound in \eqref{eqn:m0} dominates. This implies that 
we need a sufficiently high number of samples to, with high confidence, bound the event that the true change with respect to the subsequent label, $f_{m+1}$, is not considerably lower than the observed, empirical change (encoded by event $E \cap \overline{A}$). 
Intuitively, if the target is moving fast, then incorrectly predicting the label of a new sample if the empirical error is low (event $E \cap \overline{A}$) requires more samples. This is since with fewer samples we may get into a situation with a low empirical error, however, due to the target changing fast the error when it comes into predicting the label of yet another sample may be significantly higher.
With reference to Figure \ref{fig:mu_bnds}, for any fixed $\eps$, the minimum number of samples required increases as $\overline{\mu}$ increases (compare the non-constant part of the curves).

To account for both cases and make sure that the probability of both events $A$ and $E \cap \overline{A}$ is sufficiently low, we take the maximum of the associated sample size bounds. 
\end{remark}

\begin{figure}[t!]
 	\centering
	\includegraphics[width=0.5\textwidth]{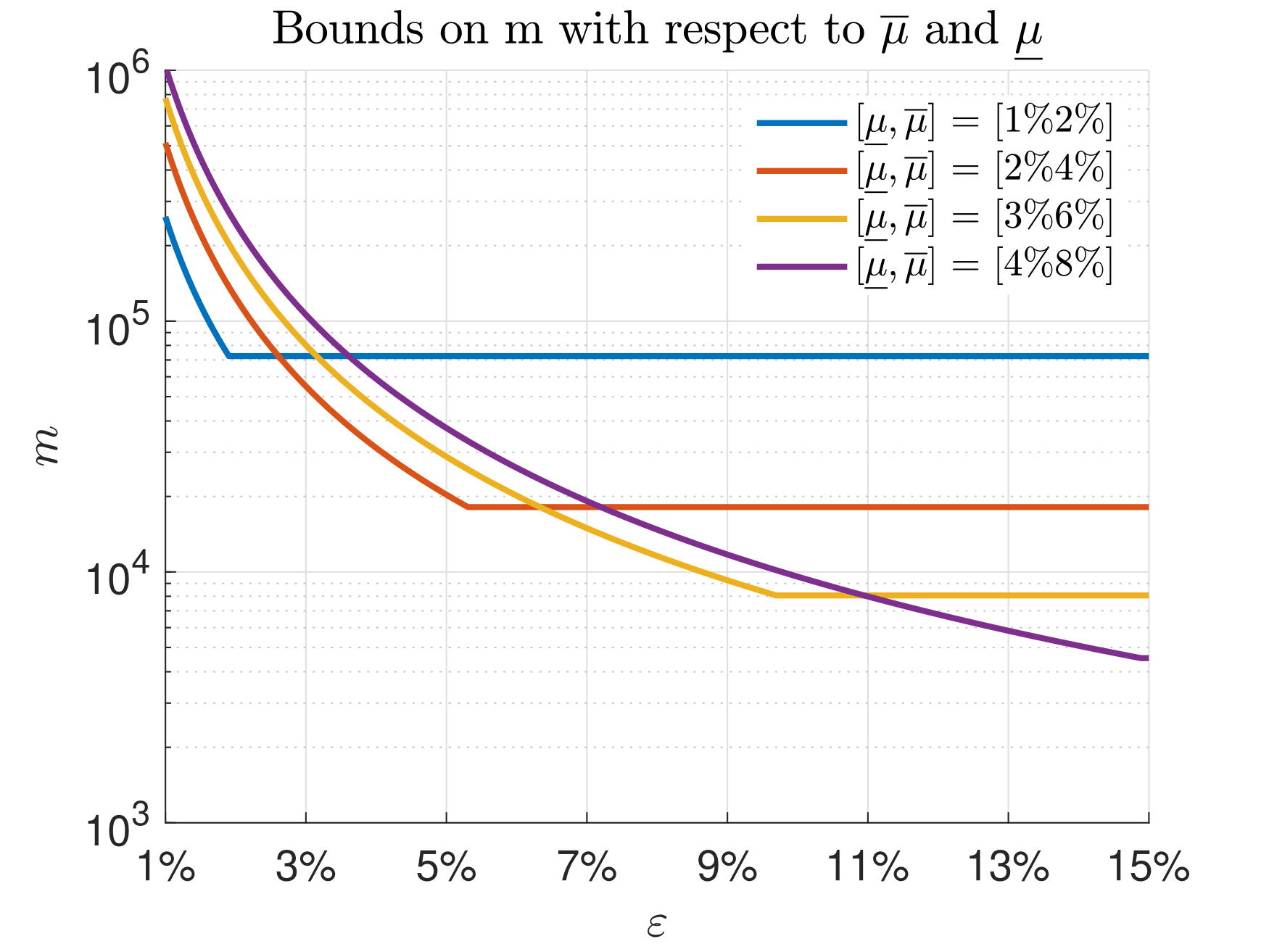}
	\caption{Number of samples required according to \eqref{eqn:m0} for different accuracy levels $\eps$ and $\delta = 10^{-6}$ with \gls{vc} dimension 4. The color code corresponds to different values of $\underline{\mu}, \overline{\mu}$. Notice that the term dependent on $\underline{\mu}$ in \eqref{eqn:m0} does not depend on $\epsilon$ and thus constitutes the constant dominant at higher levels of $\eps$.}
 \label{fig:mu_bnds}
\end{figure}

\begin{remark}[Constant target]\label{rem:constant}
The case of a constant target can be obtained as a direct byproduct of the proof of Theorem~\ref{thm:main}. To see this, notice that a constant target implies that $\underline{\mu} = \overline{\mu} = 0$, i.e., if all target functions are the same, their mutual error is zero. As such, $f_i(x_i)$ and $ f_{m+1}(x_i)$ will always be in agreement. We present the proof of Theorem~\ref{thm:main} under a constant target assumption in the Appendix. As a result, the sample size bound is identical to that of Theorem \ref{thm:1} with $\rho = 0$. This implies, that when it comes to providing guarantees for the minimal disagreement hypothesis and for the case where the target is constant, Theorem \ref{thm:main} specializes to the result of Theorem \ref{thm:1} with $\rho = 0$. With reference to Problem \ref{problem}, notice also that in this case, $\eps_0 = 4\overline{\mu} = 0$.
\end{remark}

\section{Hypothesis computation}\label{sec:4:hypothesisComputation}
We now consider the construction of the hypothesis $h_m$ that minimizes the empirical error with respect to the labeled $m$-multisample, i.e., $h_m \in M_m$. For the remainder of the paper, we will assume that the domain $\X$ is compact. Furthermore, we consider the labelling functions $f_i$, $i=1,\ldots,m$, to be defined as 
\begin{equation} \label{eqn:label_fcn}
    f_i(x) = {\1}_{B_i}(x) = \begin{cases} 1 & \text{if} \: x \in B_i \\ 0 & \text{otherwise} \end{cases},
\end{equation}
with the sets $B_i$, $i=1,\ldots,m$, being non-empty convex polytopes in $\mathbb{R}^n$, each of them having at most $n_f$ facets. 
As the hypothesis belongs to the same class with the target functions, we seek to find a convex polytope, denoted by $B_{h_m}$, such that the hypothesis $h_m$ defined as 
\begin{equation} \label{eqn:h_m}
    h_m(x) = {\1}_{B_{h_m}}(x) = \begin{cases} 1 & \text{if} \: x \in B_{h_m} \\ 0 & \text{otherwise} \end{cases},
\end{equation}
is in minimal disagreement with the observed labels. Since $B_{h_m}$ is a convex polytope with at most $n_f$ facets we represent it by means of $n_f$ linear inequality constraints as $Ax + b \leq 0$, where $A \in \mathbb{R}^{n_f \times n}$ and $b \in \mathbb{R}^{n_f}$. \sREV Denote each row-vector of $A$ (respectively, $b$) by $a_j$ ($b_j$), $j=1,\ldots, n_f$. For each $j=1,\ldots, n_f$, $a_j x +b_j=0$ denotes then a facet of $B_{h_m}$. \eREV We make this parameterization explicit by denoting the convex polytope as $B_{h_m}(A,b)$. Moreover, we assume that for each $j=1,\ldots,n_f$, $(a_j^\top, b_j) \in C_j \subset \mathbb{R}^{n_f+1}$, where $C_j$ is some arbitrarily large compact set that contains the origin in its interior. The purpose of the set will become clear in the sequel.

We show how to construct a \acrfull{milp} that returns the parameterization of $B_{h_m}$, namely $A$ and $b$, which results in a hypothesis $h_m \in M_m$. To this end, let
$I_1$ and $I_0$ be the set of sample indices for which the label is $1$ and $0$, respectively, i.e.,
\begin{align}
    I_1 & = \set{\text{$i\in \{1,\ldots,m\}$ such that } f_i(x_i) = 1}, \\
    I_0 & = \set{\text{$i\in \{1,\ldots,m\}$ such that } f_i(x_i) = 0}.
\end{align}
We instantiate the MILP that returns the minimal disagreement hypothesis in the following main steps: 

\emph{1. Disagreement with the sample indices in $I_1$.} Fix any $i \in I_1$, and let $x_i$ be the associated sample. Fix also a parameterization $A,b$ of $B_{h_m}$. If $h_m(x_i) = f_i(x_i) = 1$, i.e., the label that a hypothesis, constructed on the basis of $B_{h_m}(A,b)$, provides on $x_i$ agrees with that of $f_i$, then $x_i \in B_{h_m}(A,b)$ since $i \in I_1$. We thus have that
\begin{align}
x_i \in &B_{h_m}(A,b) \nonumber \\
&\iff a_j x_i + b_j \leq 0,~ \forall j=1,\ldots,n_f. \label{eqn:label1}
\end{align}
However, we are seeking a hypothesis that is in minimal disagreement with the samples, rather than in zero disagreement. As such, we want to allow for a certain number of incorrect labels, or equivalently, we want to allow violating the right-hand side of \eqref{eqn:label1}. Therefore, we introduce the slack variables $s_{ij} \geq 0$, $j=1,\ldots,n_f$, $i\in I_1$.
As such, for each $i\in I_1$, we consider the relaxed constraints
\begin{equation}
a_j x_i + b_j \leq s_{ij}, ~\forall j=1,\ldots,n_f. \label{eqn:label1_relax}
\end{equation}
By means of \eqref{eqn:label1} and the definition of $h_m$, enforcing \eqref{eqn:label1_relax}, implies that
\begin{align}\label{eqn:viol_I1}
    \begin{cases} 
        h_m(x_i) \neq f_i(x_i) & \text{if} \:  
                \sum_{j=1}^{n_f} s_{ij} > 0, \\
        h_m(x_i) = f_i(x_i) & \text{otherwise}.
        \end{cases}
\end{align}
In words, if $\sum_{j=1}^{n_f} s_{ij} > 0$ (which is satisfied if at least one $s_{ij}, j=1,\ldots,n_f$, is positive as the slack variables are non-negative) implies that the hypothesis $h_m$ disagrees with the target function $f_i$ on the sample $x_i$. If all slack variables are zero, then $h_m$ agrees with $f_i$ on $x_i$, $i\in I_1$.

\emph{2. Disagreement with the sample indices in $I_0$.} Fix any $i \in I_0$, and let $x_i$ be the associated sample. Fix also a parameterization $A,b$ of $B_{h_m}$. If $h_m(x_i) = f_i(x_i) = 0$, i.e., the hypothesis and the target function $f_i$ agree on $x_i$, then $x_i \notin B_{h_m}(A,b)$. This exclusion can imply that the sample $x_i$ would violate the half-space constraint encoding the facets of $B_{h_m}(A,b)$ for at least one facet. This can be written as a logical constraint; employing the developments of \cite{Bemporad1999, Morari2001}, we equivalently reformulate it to mixed-integer inequalities by introducing the binary variables $z_{ij} \in \{0,1\}$, $j=1,\ldots,n_f$, $i=1,\ldots,m$. 
Let $M_j = \sup_{x \in X, (a_j^\top,b_j) \in C_j} a_j x + b_j$, $m_j = \inf_{x \in X, (a_j,b_j) \in C_j} a_j^\top x + b_j$, $j=1,\ldots,n_f$. Note that these exist and are finite, as $\X$ and $C_j$, $j=1,\ldots,n_f$, are assumed to be compact. We then have that
\begin{align}
    x_i &\notin B_{h_m}(A,b) \nonumber \\
    &\iff \begin{cases}  a_j x_i + b_j \leq M_j(1-z_{ij}), ~\forall j = 1, \ldots, n_f, \\ 
                  a_j x_i + b_j > m_j z_{ij}, ~\forall j = 1, \ldots, n_f,\\
                  \sum_{j=1}^{n_f} z_{ij} \leq n_f-1. \label{eqn:label0}\end{cases}
\end{align}
Notice that if $z_{ij} = 0$, then the first inequality in \eqref{eqn:label0} becomes $a_j x_i + b_j \leq M_j$ (trivially satisfied by the definition of $M_j$), while the second one reduces to $a_j x_i + b_j > 0$. The latter implies then that $x_i \notin B_{h_m}(A,b)$ as it violates the constraint of its $j$-th facet. On the contrary, if $z_{ij} = 1$, then the first inequality in \eqref{eqn:label0} implies that $x_i$ is within the half-space defined by the $j$-th facet of $B_{h_m}(A,b)$ \footnote{Note that if $a_j x_i + b_j = m_j$, for $z_{ij} = 1$, the second inequality in \eqref{eqn:label0} would not be satisfied. This limiting case where $a_j x_i + b_j$ admits its lowest value is not an issue in the numerical implementation (see Remark 4) as a tolerance parameter is introduced to ``implement’’ strict inequalities numerically. Alternatively, we could choose any finite $m_j < \inf_{x \in X, (a_j,b_j) \in C_j} a_j^\top x + b_j$, $j=1,\ldots,n_f$, rather than choosing $m_j$ exactly equal to its lowest admissible value.}. For $x_i$ to be inside $B_{h_m}(A,b)$, i.e.,
$x_i \in B_{h_m}(A,b)$, this has to be the case for all $j=1,\ldots,n_f$, or equivalently $\sum_{j=1}^{n_f} z_{ij} = n_f$. This justifies the last constraint in \eqref{eqn:label0}.

Since we only seek a hypothesis in minimal (rather than in zero) disagreement with the target functions, we relax these constraints by introducing slack variables $s_{ij} \geq 0$, $j=1,\ldots,n_f$, $i\in I_0$. 
As such, for each $i\in I_0$, the associated relaxed constraints are given by
\begin{align}
\begin{cases}  a_j x_i + b_j \leq M_j(1-z_{ij}), ~\forall j = 1, \ldots, n_f, \\ 
                  a_j x_i + b_j > m_j z_{ij} - s_{ij}, ~\forall j = 1, \ldots, n_f,\\
                  \sum_{j=1}^{n_f} z_{ij} \leq n_f-1. \label{eqn:label0_relax}\end{cases}
\end{align}
Notice that we do not need to introduce a slack variable in the first inequality in \eqref{eqn:label0_relax}, as this becomes non-redundant only if $z_{ij} = 1$. In this case, however, satisfying the resulting inequality would already mean disagreeing with the target, so we do not need to relax that condition.
By means of \eqref{eqn:label0} and the definition of $h_m$, enforcing \eqref{eqn:label0_relax} leads to the same disagreement implications as in \eqref{eqn:viol_I1}.

\emph{3. Minimizing disagreements.} In view of constructing the hypothesis that is in minimal disagreement with the target functions, we need to be able to count the number of disagreements. However, if $i\in I_1$ we have a disagreement if $x_i \notin B_{h_m}(A,b)$, while if $i\in I_0$ we have a disagreement if $x_i \in B_{h_m}(A,b)$. By \eqref{eqn:viol_I1} and the discussion below \eqref{eqn:label0_relax}, disagreement happens if $\sum_{j=1}^{n_f} s_{ij} > 0$.
If we introduce the binary variable $v_{i} \in \{0,1\}$, $i=1,\ldots,m$, defined as
\begin{align}\label{eqn:vconstr}
    v_{i} & = \begin{cases} 
        1 & \text{if} \:  
                \sum_{j=1}^{n_f} s_{ij} > 0, \\
        0 & \text{otherwise}\MP{,}
        \end{cases}
\end{align}
then, the total number of disagreements that we seek to minimize is given by $\sum_{i=1}^m v_i$.

For each $j=1,\ldots,n_f$, we have assumed that $(a_j^\top,b_j) \in C_j$, where $C_j$ is compact and contains the origin in its interior. As such, $M_j>0$ and $m_j<0$ for all $j=1,\ldots,n_f$. Therefore, by \eqref{eqn:label1_relax} and the definition of $M_j$, $s_{ij} \leq M_j$, for all $i \in I_1$. Similarly, by \eqref{eqn:label0_relax} and the definition of $m_j$, $s_{ij} < -m_j$, for all $i \in I_0$. Notice that this follows from requiring the right-hand side in the second inequality of \eqref{eqn:label0_relax} to be greater than or equal to the worst-case lower bound of $a_jx_i +b_j$, namely, $m_j$, for the case where $z_{ij}=0$ that this constraint becomes nontrivial.
Summing the across $j=1,\ldots,n_f$, we obtain
\begin{align}\label{eqn:s_bounds}
\begin{cases}
\sum_{j=1}^{n_f} s_{ij} \in [0, \sum_{j=1}^{n_f}  M_j], & \text{if} \: i \in I_1\\
                \sum_{j=1}^{n_f} s_{ij} \in [0, -\sum_{j=1}^{n_f} m_j), & \text{if} \: i \in I_0.      
\end{cases}
\end{align}
The logical implication in \eqref{eqn:vconstr} is then reformulated as 
\begin{align} \label{eqn:cases_dis}
\begin{cases}
\sum_{j=1}^{n_f} s_{ij} - v_{i} \sum_{j=1}^{n_f} M_j \leq 0, &\text{if} \: i \in I_1, \\
\sum_{j=1}^{n_f} s_{ij} +  v_{i} \sum_{j=1}^{n_f} m_j < 0, &\text{if} \: i \in I_0.
\end{cases}
\end{align}
To see the equivalence between \eqref{eqn:cases_dis} and \eqref{eqn:vconstr}, consider the former inequality in \eqref{eqn:cases_dis}. 
Notice that if $\sum_{j=1}^{n_f} s_{ij} > 0$ then this implies that we must have $v_{i} \sum_{j=1}^{n_f} M_j >0$ which, since $M_j>0$ for all $j=1,\ldots,n_f$, implies that $v_i=1$. On the other hand, if $\sum_{j=1}^{n_f} s_{ij} = 0$, then \eqref{eqn:cases_dis} implies that $v_{i} \sum_{j=1}^{n_f} M_j \geq 0$. However, since we are seeking the minimal disagreement hypothesis and hence we will be minimizing $\sum_{i=1}^m v_i$, the minimum value of $v_i$ for which the previous inequality is satisfied is $v_i = 0$. A similar reasoning applies also to the equivalence between the second inequality in  \eqref{eqn:cases_dis} and \eqref{eqn:vconstr}.

\emph{4. Minimal disagreement MILP.}
The \gls{milp} that results in a hypothesis that is in minimal disagreement with respect to the target functions on the $m$-multisample, i.e., $h_m \in M_m$, is given by: 
\begin{align}
     &\minimize_{A, b, \big \{ \{z_{ij}, s_{ij}\}_{j=1}^{n_f} \big \}_{i=1}^m, \{v_{i}\}_{i=1}^m} \: \sum_{i=1}^{m} v_{i} \label{eqn:minDissagreement2} \\
    & \text{subject to } \nonumber \\
    &\forall i\in I_1:~ 
    \begin{cases} \label{eqn:con1}
    & a_j x_i + b_j \leq s_{ij},~\forall j=1,\ldots,n_f,\\
    & \sum_{j=1}^{n_f} s_{ij} - v_{i} \sum_{j=1}^{n_f} M_j \leq 0, 
    \end{cases} \\
    &\forall i\in I_0:~ 
\begin{cases} \label{eqn:con0}
a_j x_i + b_j \leq M_j(1-z_{ij}), ~\forall j = 1, \ldots, n_f, \\ 
                  a_j x_i + b_j > m_j z_{ij} - s_{ij},~\forall j = 1, \ldots, n_f,\\
                  \sum_{j=1}^{n_f} z_{ij} \leq n_f-1, \\
                  \sum_{j=1}^{n_f} s_{ij} +  v_{i} \sum_{j=1}^{n_f} m_j  < 0.
    \end{cases}
\end{align}
The constraints in \eqref{eqn:con1} correspond to \eqref{eqn:label1_relax} and the first inequality in \eqref{eqn:cases_dis}, encoding (relaxed) agreement on the sample with $i \in I_1$, and determining disagreements for this case, respectively. Similarly, the constraints in \eqref{eqn:con0} correspond to \eqref{eqn:label0_relax} and the second inequality in \eqref{eqn:cases_dis}, and admit a similar interpretation. 

The objective function $\sum_{i=1}^m v_i$ involves minimizing the total number of disagreements. We use the volume of the convex polytope parameterized by $A,b$, namely, $\mathrm{vol}(A,b)$, as a tie-break rule to single out a unique solution in case of multiple minimizers. Once the optimal $A, b$ is determined, we can construct $B_{h_m}(A,b)$, and hence $h_m$ by means of \eqref{eqn:h_m}.

\begin{remark}
Note that for the samples indexed by $i \in I_0$, the second inequality in the disagreement constraints in \eqref{eqn:label0} (and hence also \eqref{eqn:label0_relax}) are strict. From a numerical point of view, to implement these constraints we can turn them into non-strict inequalities, where following \cite{Bemporad1999} we introduce a tolerance parameter $\varrho \geq 0$, fixed to the numerical solver precision.

We can then replace the second inequality in \eqref{eqn:label0} by 
\begin{equation*}
    a_j x_i + b_j \geq \varrho + (m_j- \varrho) z_{ij}, ~\forall j=1,\ldots,n_f.
\end{equation*}
Similarly, the second inequality in \eqref{eqn:label0_relax} should be replaced by $a_j x_i + b_j \geq \varrho + (m_j- \varrho) z_{ij} - s_{ij}, ~\forall j=1,\ldots,n_f$. As a result, the second and fourth inequalities in \eqref{eqn:con0} should, respectively, become
\begin{align*}
a_j x_i + b_j &\geq \varrho + (m_j- \varrho) z_{ij} - s_{ij}, ~\forall j=1,\ldots,n_f \\
\sum_{j=1}^{n_f} s_{ij} &+ v_i\sum_{j=1}^{n_f} (m_j - \varrho) \leq 0.
\end{align*}

Once such a $\rho$ parameter is introduced, for samples indexed by $i \in I_0$,
the condition $\sum_{j=1}^{n_f} s_{ij} > 0$ is necessary but not sufficient for the hypothesis to disagree with the target. To see this, notice that if $z_{ij} = 0$ then the second inequality in \eqref{eqn:label0_relax} would become non-redundant, and result in $a_j x_i + b_j \geq \varrho-s_{ij}$. Due to the presence of $\varrho >0$, if $s_{ij}>0$ but $\varrho -s_{ij} > 0$, then $x_i$ may still be outside of a facet of the convex polytope thus agreeing with the target (recall that label agreement here means being outside $B_{h_m}(A,b)$) despite the fact that the associated slack is non-zero. As a result, the MILP in \eqref{eqn:minDissagreement2}-\eqref{eqn:con0} minimizes an upper bound on the total number of disagreements.
\end{remark}

\section{Numerical example} \label{sec:4:RevisistedExample}
\subsection{Problem set-up}
We demonstrate numerically our theoretical developments on a case study that involves \gls{aeb} systems. Furthermore, we consider the computational feasibility of the \gls{milp} and introduce an approach to discard redundant samples, thus reducing the constraints of the \gls{milp}. 

Let us consider a car driving along a road while receiving measurements of the distance $l$ to any vehicle or obstacle ahead, as well as its velocity $v$. If the braking distance at the current velocity exceeds the available distance to the car or obstacle ahead, we want the \gls{aeb} system to engage the brakes autonomously. The necessary braking distance in case of an emergency stop can be calculated by setting the braking force times the distance equal to the kinetic energy of the vehicle. Thus if 
\begin{equation}\label{eqn:kinetic_energy}
    \frac{1}{2} v^2 \frac{\mass}{F} \leq l,
\end{equation}
where $\mass$ is the vehicle mass and $F$ is the braking force, then there is a sufficient distance to the vehicle or obstacle ahead, hence the corresponding measurement is classified as safe. 
In view of \eqref{eqn:kinetic_energy} depending on $v^2$, hereafter we consider $x = (l, v^2)$ as the measurement vector.

The braking force will depend on the friction coefficient of the brakes and will deteriorate over time. Similarly, the vehicle mass will depend on the fuel, passengers, and cargo, which will also change over time. In line with our theoretical developments, we consider $x_i$, $i=1,\ldots,m$, to be independent measurements, with the index $i$ acting as a time-stamp. Let $F_i$ denote the corresponding braking force, which depends on $i$ to reflect the change of the friction coefficient, and let $\mass_i$ denote the vehicle mass, which will also depend on $i$ to reflect changes to the vehicle mass. This dependence of $F_i$ and $\mass_i$ on $i$ induces a different labeling function $f_i$. In particular, 
we label a sample $x = (l, v^2)$ by means of 
\begin{equation} \label{eqn:AEB_Oracle}
    f_i(x)  = \begin{cases} 1 & \text{if} \: \frac{1}{2} v^2 \frac{\mass_i}{F_i}  \leq l \\ 0 &\text{otherwise}.\end{cases}
\end{equation}
For the construction of the hypothesis, we collect a measurement $x_i$ after each engagement of the vehicle's brakes. In addition to obtaining $x_i$, we assume to obtain a brake performance measurement $\frac{\mass_i}{F_i}$ from which to construct the label $f_i$. Furthermore, we assume that we have knowledge of the expected minimum and maximum degradation of the braking performance, allowing us to obtain values for $\underline{\mu}$ and $\overline{\mu}$, respectively.\eREV 

Using numerical values for the braking parameters and vehicle mass as defined in the sequel, the evolution of the braking performance is shown in Figure~\ref{fig:concept_drift}. For visual clarity, we only depict a random subset of the samples.

\sREV For an effective \gls{aeb} system, we want to classify a new sample $x$ as safe or unsafe without having to first engage the brakes to receive a measurement of the new braking performance, which depends on the unknown braking force $F_{m+1}$ and mass $\mass_{m+1}$. We will utilize the results from Section~\ref{sec:3:FiniteSampleProbabilistic} to construct a hypothesis on the basis of a labeled $m$-multisample, which allows us to \emph{a-priori} classify a sample $x$ as safe or unsafe; this illustrates the merit of the proposed method. \eREV

\begin{figure*}[ht]
 	\centering
	\includegraphics[width=\textwidth]{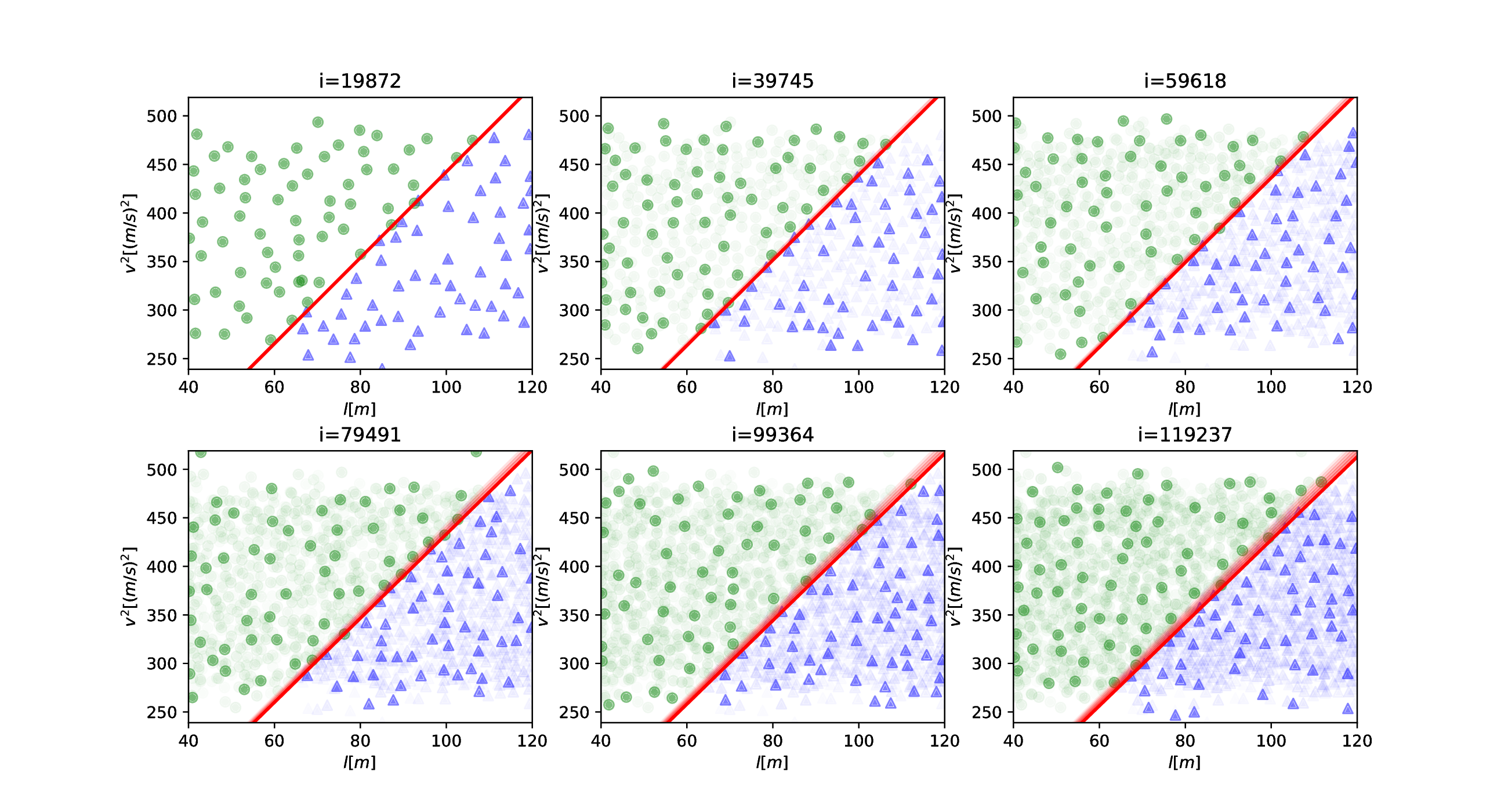}
	\caption{The evolution of the braking performance over time. Green circles indicate samples with label 0, while blue triangles show samples with label 1. The bold red halfplane represents the true safety label at the given iteration, while the opaque halfplanes show the safety boundary at previous iterations. }\label{fig:concept_drift}
\end{figure*}

Aligned with the theoretical developments of Section \ref{sec:4:hypothesisComputation} we consider our hypothesis to be a convex polytope in the 2D plane as illustrated in Figure~\ref{fig:hyperrectangle}. We assume that matrix $A$ parameterizing the convex polytope is fixed, and is given by
\begin{align}
    A = \begin{bmatrix}
         \cos{\theta} &  \sin{\theta}\\
        -\sin{\theta} &  \cos{\theta}\\
        -\cos{\theta} & -\sin{\theta}\\
         \sin{\theta} & -\cos{\theta}
    \end{bmatrix},
\end{align}
where $\theta \coloneqq \tan^{-1}{\frac{\mass}{2 F}}$ denotes the rotation of the convex polytope. Since the safety label is defined by a single half-plane, only one of the facets of the convex polytope becomes relevant, namely $a_3 = [-\cos{\theta} \:  -\sin{\theta}]$. This observation reduces the \gls{vc}-dimension (employed in the sample complexity bounds) to $d=1$. Considering the evolution of the braking performance as shown in Figure~\ref{fig:concept_drift}, the rotation of the convex polytope is minimal. Since the inclusion of the variable rotation introduces a nonlinearity into the \gls{milp}, for the sake of clarity, we will consider the angle $\theta$ to be fixed in the subsequent computation of the hypothesis (however leave the true safety label unchanged).

\begin{figure}[ht]
 \includegraphics[width=1.09\columnwidth]{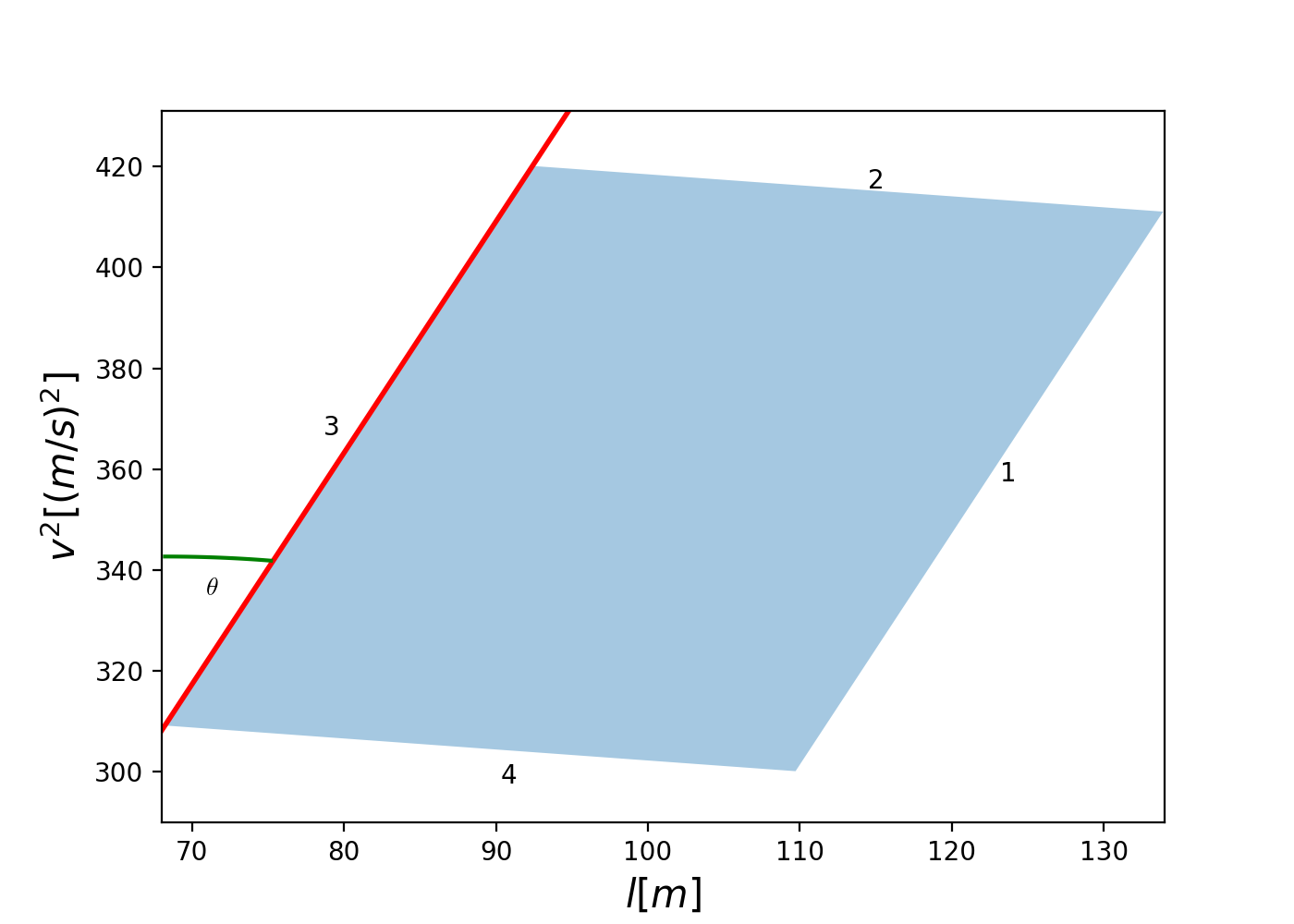}
	\caption{Illustration of the facets of the convex polytope. Since the safety label relies in this case on a single halfplane (drawn in red), we only need to consider facet 3.}\label{fig:hyperrectangle}
\end{figure}

To reduce the size of the \gls{milp} (\ref{eqn:minDissagreement2} -~\ref{eqn:con0}), we will consider how samples can be discarded prior to computing the hypothesis. Recall that the \gls{milp} minimizes the total number of disagreements between the hypothesis and the sample labels. Thus, at best, we can obtain zero disagreement between the hypothesis and the samples in $I_1$. For the \gls{aeb} example, this implies that the halfplane constructed by the hypothesis will need to lie to the left of all samples in $I_1$. This is illustrated in Figure~\ref{fig:discards} by the cyan dotted line. However, any halfplane that lies further to the left of this line, would unnecessarily label samples from $I_0$ with label \emph{one}, increasing the total number of label disagreements. However, the constructed \gls{milp} aims at minimizing the total number of label disagreements; as such the cyan dotted halfplane would always be preferable, rendering any samples in the cyan-colored area redundant. If we have non-zero disagreement with respect to the samples from $I_1$, the halfplane will lie further to the right of the cyan dotted one.
Similarly, to obtain zero disagreement between the hypothesis and the samples in $I_0$, the halfplane constructed by the hypothesis will need to lie to the right of all samples in $I_0$. Following a similar argumentation as before, any sample in the magenta-colored area will not change the solution of the \gls{milp}.

Since we know that the samples in both the blue and the magenta regions will not affect the solution of the \gls{milp}, we can discard these samples prior to computing the hypothesis, resulting in only the red samples in Figure~\ref{fig:discards} being considered.
Following similar arguments, it can be possible to discard redundant samples also in the setting of higher-order convex polytopes. However, generalizing the proposed methodology to achieve this is case-dependent and is not pursued further here.

\subsection{Simulation results} \label{sec:4:part2}
While no knowledge of the distribution of the samples $(l, v^2)$ needs to be known for generating the hypothesis, for simulation purposes we draw $l$ from a uniform distribution over the interval $[40m, 120m]$ and draw $v^2$ from a normal distribution with mean $\overline{v^2} = {(70 km/h)}^2$ and standard deviation $\sigma_{v^2} = {(20 km/h)}^2$. The performance of the brakes at each time step will deteriorate by a factor of $\omega_F$, i.e. $F_{i+1} = \omega_F  F_{i}$, where $\omega_F$ is a random variable drawn from a normal distribution with mean $\mu=(1-3 \cdot 10^{-7})$ and standard deviation $\sigma = 10^{-6}$. The initial car mass is $\mass = 900 kg$ and will randomly change by a factor of $\omega_\mass$, where $\omega_\mass$ is a random variable drawn from a normal distribution with mean $\mu=1$ and standard deviation $\sigma = 10^{-3}$.

For the construction of the hypothesis, the confidence level is chosen as $\delta=10^{-6}$ with an accuracy of $\epsilon= 1\%$. For the satisfaction of Assumption~\ref{assum:mu}, we choose $\frac{1}{m}\sum_{i=1}^{m} \er{f_i}{f_{m+1}}$ to be bounded by $\overline{\mu} \leq 2\%$ and $\underline{\mu} \geq 0.78\%$. By Theorem~\ref{thm:main} it then follows that we need at least $119,237$ samples to accurately predict the safety label of the subsequent timesteps.

Using the aforementioned discarding approach, we can discard $95\%$ of the samples prior to instantiating the \gls{milp} constraints. The discarding approach is illustrated in Figure~\ref{fig:discards} where, for the purpose of visualization, we omit samples close to one another to prevent the image from being cluttered.
\begin{figure}[ht]
 	\includegraphics[width=1.09\columnwidth]{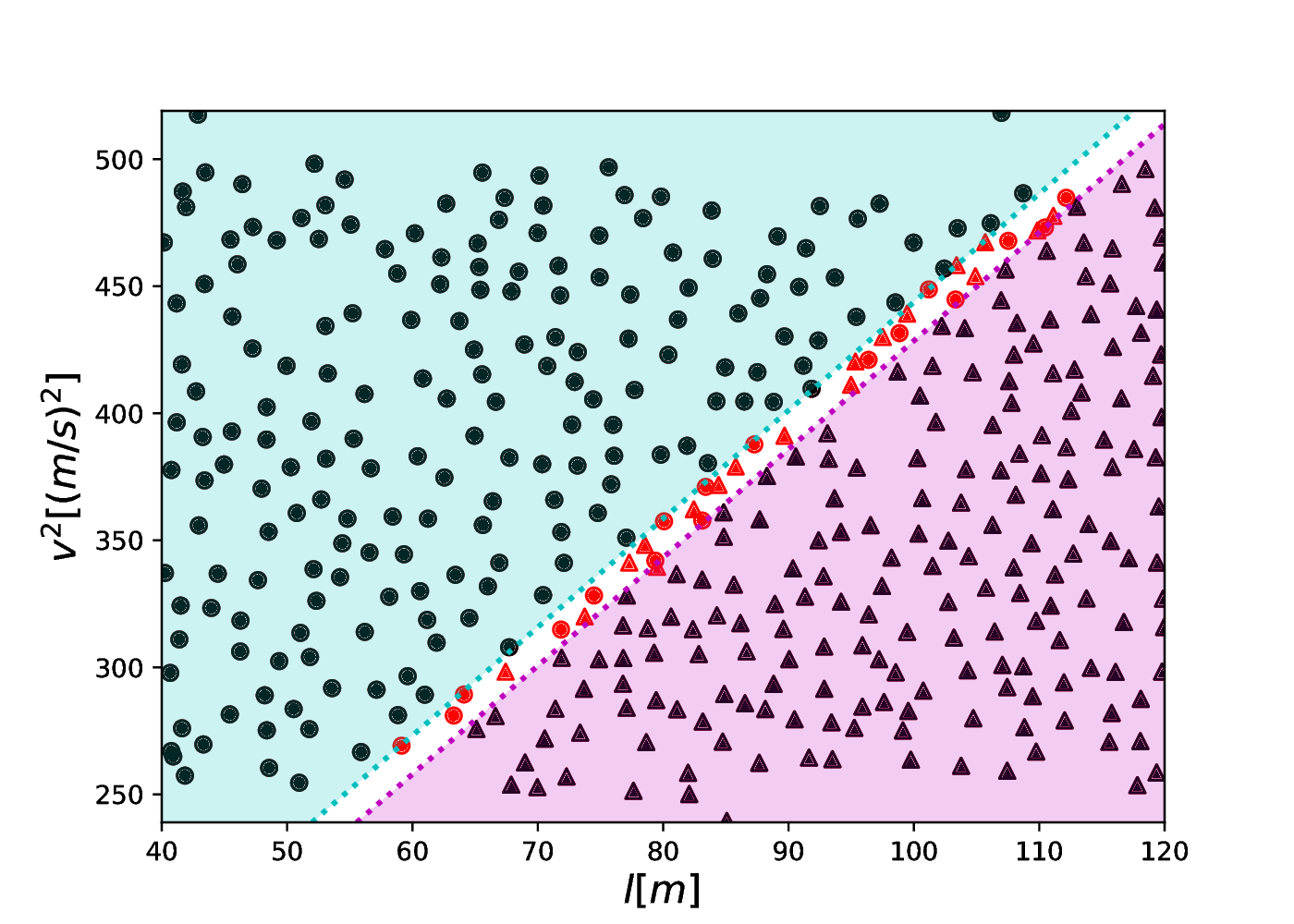}
	\caption{All samples in black are discarded, while the red samples are kept for the computation of the hypothesis. This results in $95\%$ of the samples being omitted, greatly improving the computational feasibility.}\label{fig:discards}
\end{figure}
The hypothesis in minimal disagreement with the labeled samples, computed by means of the \gls{milp} (\ref{eqn:minDissagreement2} -~\ref{eqn:con0}), is shown in Figure~\ref{fig:hypothesis}. Solving the \gls{milp} took 561 seconds, making the deployment of the approach computationally feasible. The number of violations, $v$, is 1335. We have made all code for generating and reproducing our results available online\footnote{\url{https://github.com/nikovert/lrn-moving-targets}}. 

We empirically validate our risk level by means of Monte Carlo simulations. For each run, we generate a new labeling mechanism $f_{m+1}$, corresponding to the random deterioration of the braking force and change to the vehicle mass. We then draw 5000 samples for which we evaluate the corresponding label (by means of $f_{m+1}$) and compare this with the label assigned by means of the hypothesis constructed by our methodology, thus calculating $\erhat{f_{m+1}}{h_m}$. We repeat this for 500 runs, each time generating a new label $f_{m+1}$. In Figure~\ref{fig:montecarlo} the frequency of certain $\erhat{f_{m+1}}{h_m}$ values is shown. 

Recall that $\mu$ was upper bounded by $2\%$, such that for the chosen $\epsilon=1\%$, Theorem \ref{thm:main} implies that $\er{f_{m+1}}{h_m} \leq 4\overline{\mu}+\eps = 9 \%$ with high confidence. The Monte Carlo simulation supports this, with the average empirical disagreement $\erhat{f_{m+1}}{h_m}$ being approximately $2.4\%$ (see Figure~\ref{fig:montecarlo}), well below the theoretically predicted $9 \%$.

\begin{figure}[ht]
 	\includegraphics[width=1.09\columnwidth]{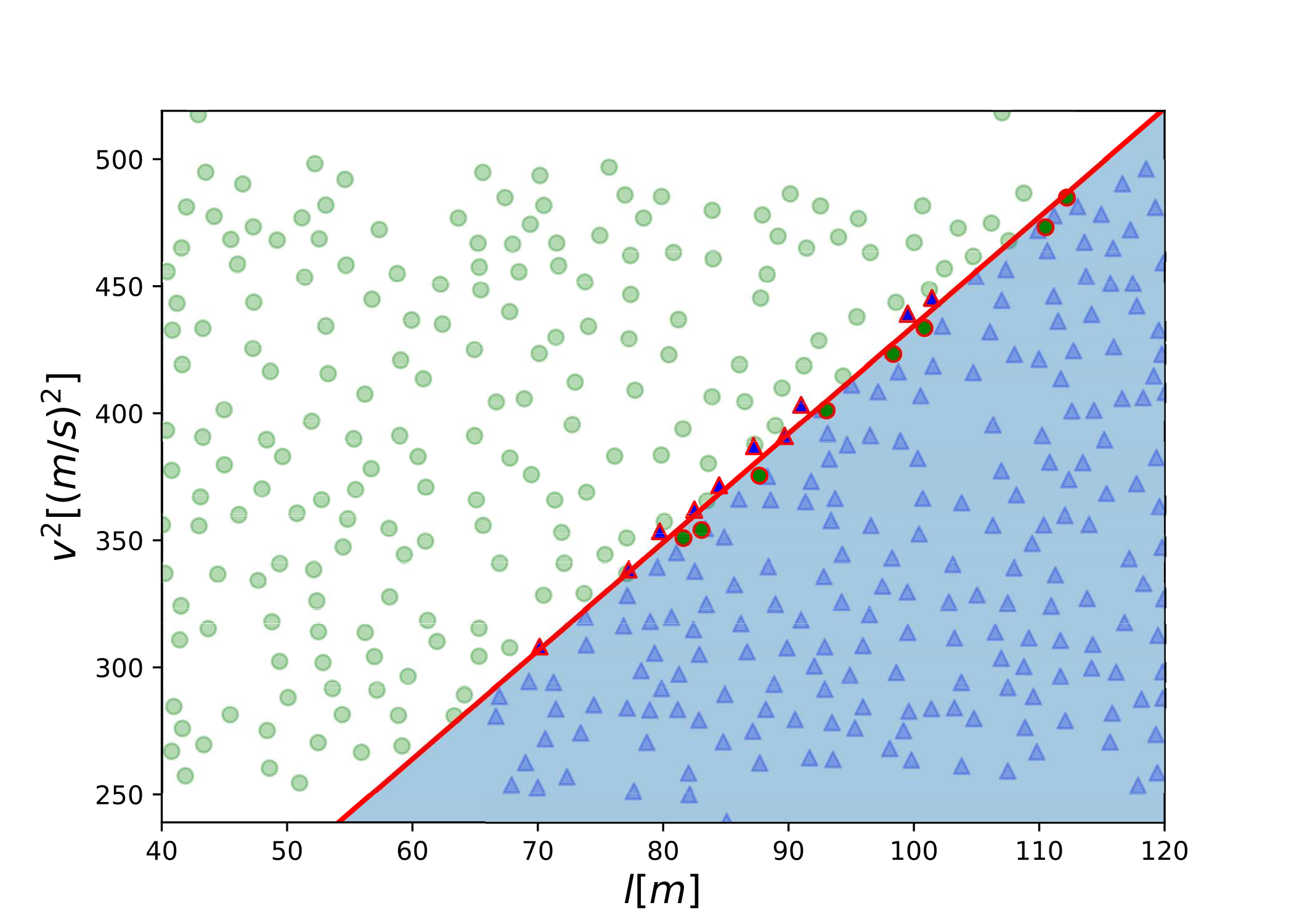}
	\caption{Generated hypothesis; we only show the halfplane responsible for the labeling, illustrated by red. For visual ease, we randomly omit samples close to one another. Red samples are violations as defined in~\eqref{eqn:vconstr}.}\label{fig:hypothesis}
\end{figure}

\begin{figure}[ht]
    \centering
    \includegraphics[width=\columnwidth]{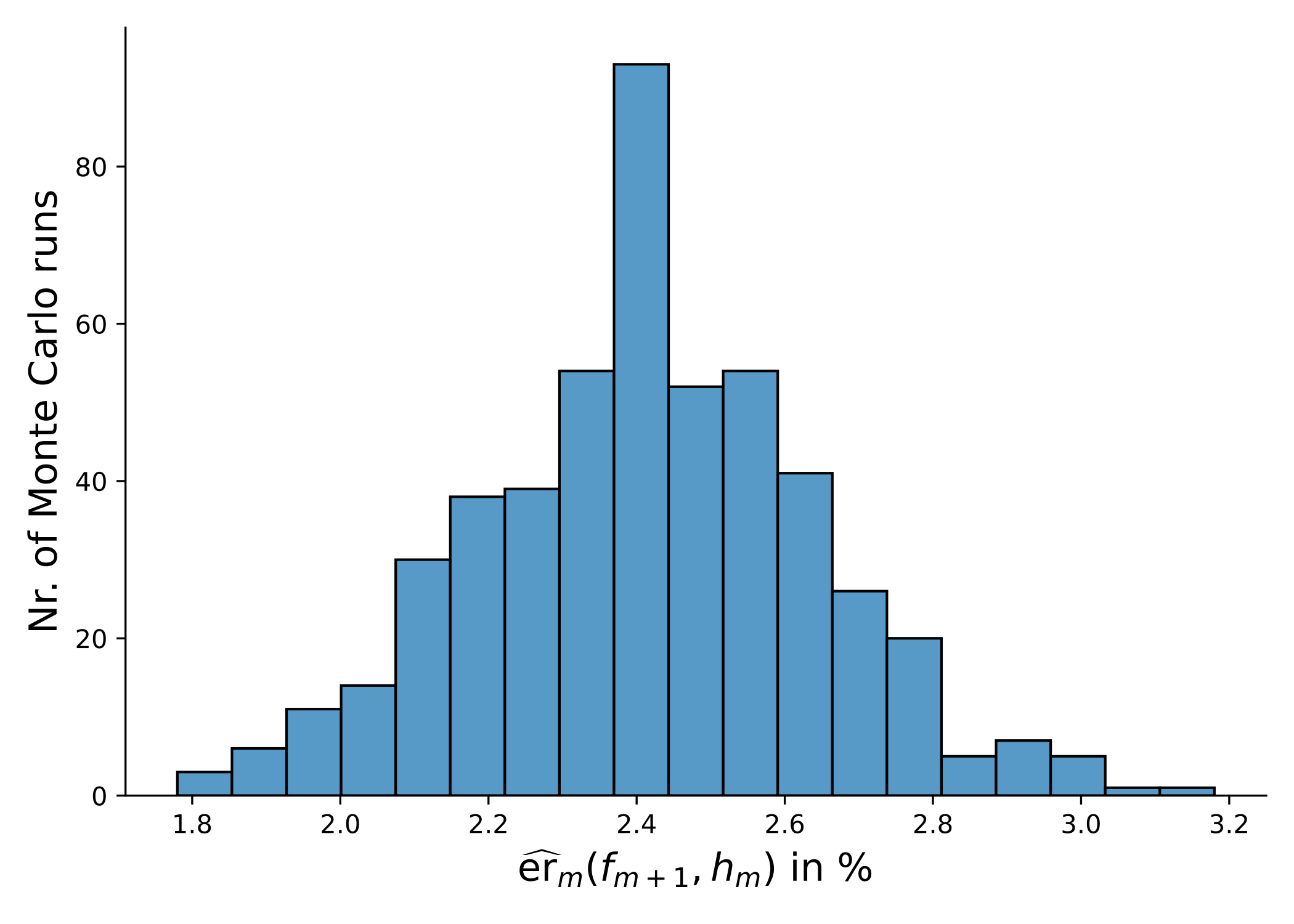}
	\caption{Empirical distribution of the disagreement $\erhat{f_{m+1}}{h_m}$, constructed by means of 500 Monte Carlo runs. }\label{fig:montecarlo}
\end{figure}

\section{Conclusion}\label{sec:5:conclusion}
We considered learning a moving target from a finite set of samples and showed that, when the labeling mechanism changes in a structured manner, it remains \gls{pac} learnable, meeting certain accuracy-confidence levels. Furthermore, for the class of convex polytopes, we presented a constructive method to generate the hypothesis based on a \acrfull{milp}. We illustrated the applicability of our theoretical developments to a case study involving an \acrfull{aeb} system. Future work aims at considering the distribution according to which samples are drawn to also be changing, similarly to \cite{Long1999}.

\appendix 
\section{Recovering the constant target case}
Following Remark \ref{rem:constant}, we show in the next result how Theorem \ref{thm:main} specializes to obtain probabilistic guarantees for a minimal disagreement hypothesis $h_m \in M_m$, for the case where the target is constant.
\begin{thm}
    Fix $\eps, \delta \in (0,1)$. 
    Denote by $d$ the \gls{vc} dimension of $\HH$, and consider $\underline{\mu} = \overline{\mu} = 0$.
    If we choose $m \geq m_0(\eps,\delta)$, where
    \begin{align}
            m_0(\eps,\delta) = \frac{5}{\eps} \Big( \ln\frac{4}{\delta} + d \ln\frac{40}{\eps}\Big), \label{app_app_eqn:m0}
    \end{align}
    we then have that for any $h_m \in M_m$,
    \begin{align}
        \mathbb{P}^m\{(x_1,\ldots,x_m) &\in \X^m:~ \nonumber\\
        &\er{f_{m+1}}{h_m} \leq \eps \} \geq 1-\delta. \label{app_eqn:thm2}
    \end{align}
\end{thm}

\begin{proof}
Fix any $\eps, \delta \in (0,1)$. We follow the same proof-line with Theorem \ref{thm:main}, but since $\underline{\mu} = \overline{\mu} = 0$, all target functions are identical. To this end, let $f_i = f$, for all $i=1,\ldots,m,m+1$.
We define the following event:
\begin{align}
&E = \{(x_1,\ldots,x_m) \in \X^m :~ \er{f}{h_m} > \eps\}, \nonumber \\
&\widehat{E} = \{(x_1,\ldots,x_m) \in \X^m :~ \erhat{f}{h_m} = 0\}.
\end{align}
$\widehat{E}$ is the set of $m$-multisamples for which the empirical average disagreement between the (constant) target and the hypothesis, namely, $\frac{1}{m}\sum_{i=1}^m |f(x_i)- h_m(x_i)|$, is equal to zero.
Notice that since $h_m \in M_m$ (a minimal disagreement hypothesis), for any $m$-multisample, $\sum_{i=1}^m |f(x_i)- h_m(x_i)| \leq \sum_{i=1}^m |f(x_i)- h(x_i)|$ for any $h\in \HH$. Since the target function $f$ itself is an element of $\HH$, taking $h = f$ in the aforementioned statement directly leads to $\sum_{i=1}^m |f(x_i)- h_m(x_i)| \leq 0$, and hence $\Prb^m\{\widehat{E}\} = 1$.

To establish \eqref{app_eqn:thm2} it suffices to show that $\Prb^m\{E\} \leq \delta$. To this end, we have that
    \begin{align}
        & \Prb^m\{E\} \nonumber \\
        & = \Prb^m\{E \cap \widehat{E}\}\nonumber \\
        &=\Prb^m\{(x_1,\ldots,x_m)  \in \X^m:~
        \er{f}{h_m} > \eps \nonumber \\
        & \hspace{1.5cm}\text{and } \erhat{f}{h_m} =0 \}.
        \label{app_eqn:constantT}
    \end{align}
where the first equality is since $\Prb^m\{\widehat{E}\} = 1$, and the second one follows from the definition of $E$ and $\widehat{E}$. 

Notice that \eqref{app_eqn:constantT} takes the form of \eqref{eqn:thm:1}, with $f_{m+1}$, $h_m$ and $0$ in place of $f$, $h$ and $\rho$, respectively. Theorem \ref{thm:1} implies then that 
\begin{align}
m &\geq \frac{5}{\eps} \Big( \ln\frac{4}{\delta} + d \ln\frac{40}{\eps}\Big) 
\Longrightarrow~ \Prb^m\{E \cap \widehat{E}\} \leq \delta. \label{app_eqn:bound2} 
\end{align}
Therefore, by \eqref{app_eqn:constantT} and \eqref{app_eqn:bound2}  we have that $\Prb^m\{E\} \leq \delta$, thus concluding the proof.
\end{proof}

\eREV

\bibliographystyle{abbrv}        %
\bibliography{bibliography}           %

\end{document}